\documentclass[12pt]{amsart}
\usepackage[mathscr]{eucal}
\usepackage{amssymb}
\usepackage{amsfonts}
\usepackage{ytableau}

\theoremstyle{plain} 
\newtheorem{thm}{Theorem}[section]
\newtheorem{cor}[thm]{Corollary}
\newtheorem{lem}[thm]{Lemma}
\newtheorem{prop}[thm]{Proposition}

\theoremstyle{definition}

\newtheorem{ex}[thm]{Example}
\newtheorem{conj}[thm]{Conjecture}
\newtheorem{prob}[thm]{Problem}

\theoremstyle{remark}
\newtheorem{rem}[thm]{Remark}

\numberwithin{equation}{section}

\def\NN{{\mathbb N}}

\def\CC{{\mathbb C}}
\def\AA{{\mathbb A}}

\def\fp{{\mathfrak p}}
\def\fq{{\mathfrak q}}
\def\fm{{\mathfrak m}}

\def\ba{{\mathbf a}}
\def\bb{{\mathbf b}}

\def\Spec{\operatorname{Spec}}
\def\height{\operatorname{ht}}
\def\Ass{\operatorname{Ass}}

\def\Hilb{\operatorname{Hilb}}
\def\ann{\operatorname{Ann}}

\def\too{\longrightarrow}

\def\Min{\operatorname{Min}}

\def\Im{\operatorname{Im}}

\def\StTab{\operatorname{StTab}}
\def\Tab{\operatorname{Tab}}
\def\supp{\operatorname{supp}}

\def\ISp{I^{\rm Sp}}

\def\ba{\mathbf a}

\def\frJ{{\frak J}}

\def\too{\longrightarrow}

\def\bP{\overline{P}}

\def\chara{\operatorname{char}}

\def\depth{\operatorname{depth}}
\def\<{{\langle}}
\def\>{{\rangle}}

\begin{document}
\title[When is a Specht ideal Cohen--Macaulay?]{When is a Specht ideal Cohen--Macaulay?}
\author{Kohji Yanagawa}
\address{Department of Mathematics, Kansai University, Suita, Osaka 564-8680, Japan}
\email{yanagawa@kansai-u.ac.jp}
\thanks{The author is partially supported by JSPS Grant-in-Aid for Scientific Research (C) 16K05114.}
\maketitle

\begin{abstract}
For a partition $\lambda$ of $n$, let $\ISp_\lambda$  be the ideal of $R=K[x_1, \ldots, x_n]$ generated by all Specht polynomials of shape $\lambda$. 
We show that if $R/\ISp_\lambda$ is Cohen--Macaulay then $\lambda$ is of the form either $(a, 1, \ldots, 1)$,  $(a,b)$, or  $(a,a,1)$. We also prove that the converse is true in the $\chara(K)=0$ case. 
To show the latter statement,   the radicalness of these ideals and a result of Etingof et al. are crucial. We also remark that  $R/\ISp_{(n-3,3)}$ is {\it not} Cohen--Macaulay if and only if $\chara(K)=2$.
\end{abstract}

\section{Introduction}
Let $n$ be a positive integer. A {\it partition} of $n$ is a sequence $\lambda = (\lambda_1, \ldots, \lambda_l)$ of intergers with $\lambda_1 \ge \lambda_2 \ge \cdots \ge \lambda_l \ge 1$ and $\sum_{i=1}^l \lambda_i =n$.  
A partition $\lambda$ is often represented by its {\it Young diagram}.  The {\it (Young) tableau} of shape $\lambda$ is a bijection from $[n]:=\{1,2, \ldots, n\}$ to the set of boxes in the Young diagram of $\lambda$. 
 For example, the following is a tableau of shape  $(4,2,1)$. 
\begin{equation}\label{ex of T}
\begin{ytableau}
3 & 5 & 1& 7   \\
6 & 2    \\
4 \\
\end{ytableau}
\end{equation}
We say a tableau $T$ is {\it standard}, if all columns (resp. rows) are increasing from top to bottom (resp. from left to right). 

Let $R=K[x_1, \ldots, x_n]$ be a polynomial ring over a field $K$, $\lambda$ a partition of $n$, and $T$ a Young tableau of shape $\lambda$.  Now let $f_T(j)$ denote the difference product of the variables whose subscripts  belong to the $j$-th column of $T$. 
More precisely, if the $j$-th column of $T$ consists of $j_1, j_2, \ldots, j_m$ in the order from top to bottom, then 
$$f_T (j) = \prod_{1 \le s < t \le m} (x_{j_s}-x_{j_t})$$
(if the $j$-th column has only one box, then we set $f_T(j)=1$).  Finally, we set 
$$f_T := \prod_{j=1}^{\lambda_1} f_T(j),$$
and call it the {\it Specht polynomial} of $T$.   For example, if $T$ is the tableau given in \eqref{ex of T}, then $f_T=(x_3-x_6)(x_3-x_4)(x_6-x_4)(x_5-x_2).$ 

If $\lambda$ is a partition of $n$, the symmetric group $S_n$ acts on the vector space $U_\lambda$ of $R$ spanned by 
$$\{ \, f_T  \mid \text{$T$ is a Young tableau of shape $\lambda$ } \}.$$
(Literature construct  $U_\lambda$ using  suitable equivalent classes of Young tableaux, not polynomials in $R$.  For example,  \cite{Sa} uses the notion called {\it Young polytabloids}.  Of course, such a  construction gives the same modules as ours up to isomorphism.) 
 An $S_n$-module of this form is called a {\it Specht module}, and very important in the  theory of  symmetric groups.  In fact,  if $\chara(K)=0$, they give a complete list of irreducible representations of $S_n$, if we consider all possible partitions $\lambda$ of $n$.

Specht modules occasionally appear in the study of commutative algebra. For example, 
\cite{LefPro} decomposes the artinian ring $R/(x_1^d, \ldots, x_n^d)$ into direct sum of Specht modules under the natural $S_n$-action. However, the present paper concerns the {\it ideal}  
$$\ISp_\lambda := (U_\lambda) =(\, f_T  \mid \text{$T$ is a Young tableau of shape $\lambda$ })$$
 of the polynomial ring $R$ itself (throughout the paper, we exclude the trivial partition $(n)$ of $n$, while some results make sense if we put $\ISp_{(n)} =R$). We focus on the following problem.

\begin{prob}\label{main prob}
When is $R/\ISp_\lambda$ Cohen--Macaulay?
\end{prob}

For  many $\lambda$, $R/\ISp_\lambda$ is even non-pure (i.e., minimal primes have different dimensions).  
Moreover, for some $\lambda$, the Cohen--Macaulay property of $R/\ISp_\lambda$ depends on $\chara(K)$.   For example, $R/\ISp_{(n-3,3)}$ is Cohen--Macaulay if and only if $\chara(K) \ne 2$ (Theorem~\ref{3-ji shiki}). 

For $\lambda=(\lambda_1, \ldots, \lambda_l)$, it is not difficult to see that $\height(\ISp_\lambda) = \lambda_1$ (see Lemma~\ref{lambda_1+1}). 
Moreover, in Proposition~\ref{CM ISP}, we see that if $R/\ISp_\lambda$ is Cohen--Macaulay then one of the following conditions  is satisfied. 
\begin{itemize}
\item[(1)] $\lambda =(n-d, 1, \ldots, 1)$, 
\item[(2)] $\lambda =(n-d,d)$, 
\item[(3)] $\lambda =(a,a,1)$. 
\end{itemize}

If $\lambda =(n-d, 1, \ldots, 1)$, then $\ISp_\lambda$ is generated by all maximal minors of a $(d+1) \times n$ Vandermonde matrix,  and Junzo Watanabe and the author (\cite{WY}) showed that  $R/\ISp_\lambda$ is reduced and Cohen--Macaulay in this case. So it remains  to consider the cases (2) and (3). In these cases, we have 
$$\sqrt{\ISp_\lambda}= \bigcap_{\substack{F \subset [n] \\ \# F = \lambda_1+1}} (x_i-x_j \mid i, j \in F).$$ 
Etingof et al. \!\!\cite{EGL} studied such ideals, and their result states that $R/\sqrt{\ISp_\lambda}$ is Cohen--Macaulay if $\chara(K)=0$ (in the cases (2) and (3)).


On the other hand, in Theorems~\ref{ISp_{(n-k,k)} is radical} and \ref{ISp_{(a,a,1)} is radical}, we show  that $R/\ISp_{(n-d,d)}$ and $R/\ISp_{(a,a,1)}$ are reduced. Hence they are Cohen--Macaulay if $\chara(K)=0$ (by virtue of a result of \cite{EGL}). Summing up,  in the characteristic 0 case, $R/\ISp_\lambda$ is Cohen--Macaulay if and only if one of the above conditions (1), (2) or (3) is satisfied.  

Another application of the radicalness of $\ISp_{(n-d,d)}$ is a new characterization of Catalan numbers. 
More precisely, in the polynomial ring $K[x_1, \ldots, x_{2n}]$, consider the ideal 
$$I:=  \bigcap_{\substack{F \subset [2n] \\ \# F = n+1}} (x_i-x_j \mid i, j \in F).$$
Since $I= \sqrt{\ISp_{(n,n)}}= \ISp_{(n,n)}$, we have $\mu(I) =C_n$, where $C_n$ is the $n$-th Catalan number $\displaystyle \frac{1}{2n+1}\binom{2n+1}{n}$.


\section{Minimal primes  and the necessity condition for the CM-ness}
Let $R=K[x_1, \ldots, x_n]$ be a polynomial ring  over a field $K$. 
We assume that $K$ is algebraically closed for the simplicity. However,  for all results of the present paper, this assumption can be easily dropped.  

For an ideal $I \subset R$, set $V(I):=\{ \fp \mid \fp \in \Spec R,  \fp \supset I\}$ as usual. 
For $\ba=(a_1, \ldots, a_n) \in K^n$, let $\fm_\ba$ denote the maximal ideal 
$(x_1 -a_1, x_2 - a_2, \ldots, x_n - a_n)$ of $R$.  
By abuse of notation, we just write $\ba \in V(I)$ to mean $\fm_\ba \in V(I)$. 
Note that $\ba \in V(I)$, if and only if $\ba$ belongs to the algebraic subset 
of $\AA^n$ defined by $I$, if and only if $f(\ba)=0$ for all $f \in I$. 

For the definition of  the Specht ideal $\ISp_\lambda$, see the previous section. 
Let us begin to study the Cohen--Macaulay property of $R/\ISp_\lambda$. 
If $\lambda_2=1$, then  $\ISp_\lambda$ is the determinantal ideal of a  Vandermonde like  matrix, and  Watanabe and the author (\cite{WY}) showed that  $R/\ISp_\lambda$ is reduced and Cohen--Macaulay in this case. 
So we mainly treat the case $\lambda_2 > 1$ in this paper.

\medskip

The following fact immediately follows from the definition. 

\begin{lem}\label{0 point}
For $\ba =(a_1, \ldots, a_n) \in K^n$, $\ba  \in V(\ISp_\lambda)$ if and only if $\ba$ satisfies the following condition; 

\noindent$(*)$ For any tableau $T$ of shape $\lambda$, there exist two distinct integers $i, j \in [n]$ such that $a_i=a_j$ and $i,j$ appear in the same column of $T$.   
\end{lem}

Let $\Pi = \{ F_1, \ldots, F_m \}$ be a partition of $[n] := \{1,2, \ldots, n\}$, that is, $[n]=\coprod_{i=1}^m F_i$ and $F_i \ne \emptyset$ for all $i$.  We call the ideal 
$$P_\Pi := ( \, x_i -x_j \mid \text{$i,j \in F_k$ for $k=1,2, \ldots, m$} \, ) \subset R$$ 
the {\it partition ideal} of $\Pi$. Clearly, this is a prime ideal with $R/P_\Pi \cong K[X_1, \ldots, X_m]$. 
Hence we have $\dim R/P_\Pi = m$ and  $\height(P_\Pi)=n-m$.  
It is easy to see that if an ideal $I \subset R$ is generated by elements of the form $x_i -x_j$ then $I=P_\Pi$ for some partition $\Pi$ of $[n]$. 

\begin{lem}\label{partition} 
A minimal prime of $\ISp_\lambda$ is the partition ideal  $P_\Pi$ of some $\Pi$. 
\end{lem}

\begin{proof}
Since $V((f_T))$ for a tableau $T$ is a union of hyperplanes of the form $V(x_i-x_j)$ for some  $i \ne j$, 
the irreducible components of $V(\ISp_\lambda)$ are intersections of these hyperplanes, that is, linear spaces of the form $V(P_\Pi)$ for some $\Pi$. 
\end{proof}

For a subset $F \subset [n]$ with $\#F \ge 2$,
consider the prime ideal $$P_F = ( \, x_i -x_j \mid i,j \in F \, )$$ of $R$. 
Clearly, $P_F$ is a special case of  partition ideals, and we have  $\height(P_F) = \# F -1$. 

\begin{prop}\label{lambda_1+1}
Let $\lambda=(\lambda_1, \ldots, \lambda_l)$ be a partition of $n$. 
For a subset $F \subset [n]$ with $\#F = \lambda_1+1$, $P_F$ is a minimal prime of  $\ISp_\lambda$. 
Moreover, we have $\height(\ISp_\lambda) = \lambda_1$. 
\end{prop}

\begin{proof}
First, we will show that $P_F \supset \ISp_{\lambda}$ for $F \subset [n]$ with $\# F = \lambda_1+1$. It suffices to show that any $\ba \in  V(P_F)$ satisfies $\ba \in V(\ISp_\lambda)$. However, if $\ba \in  V(P_F)$, then $a_i =a_j$ for all $i. j \in F$, and it satisfies the condition $(*)$ of Lemma~\ref{0 point} by the pigeonhole principle. 

Next,  consider a partition $\Pi=\{ F_1, \ldots, F_d \}$ of $[n]$. 
Set $c_i= \# F_i$ for each $i$. And we may assume that $c_1 \ge c_2 \ge \cdots$. 
We will show that if  $\height(P_\Pi) < \lambda_1$ 
(equivalently, $d > n-\lambda_1 = \lambda_2 +\cdots + \lambda_l$) then $P_\Pi \not \supset \ISp_\lambda$.  
We can take $\ba=(a_1, \ldots, a_n) \in V(P_\Pi)$ so that $a_i =a_j$ if and only if $i,j \in F_k$ for some $k$. 
To show $P_\Pi \not \supset \ISp_\lambda$, it suffices to prove that $\ba \not \in V(\ISp_\lambda)$, equivalently, $\ba$ does not satisfy the condition $(*)$ of Lemma~\ref{0 point}. This is also equivalent to that we can assign one of $d$-colors  to each box in the Young diagram of $\lambda$ so that the number of boxes painted in the $i$-th color  is exactly $c_i$, and 
any two boxes in the same column have different colors.  The existence of such coloring is illustrated in the following  way (here, integers represent colors of boxes). 
First, we paint the boxes in the second  line to the last line as follows. 
In other words, we paint these boxes just like counting them in the ``western letter-writing" order.  
$$
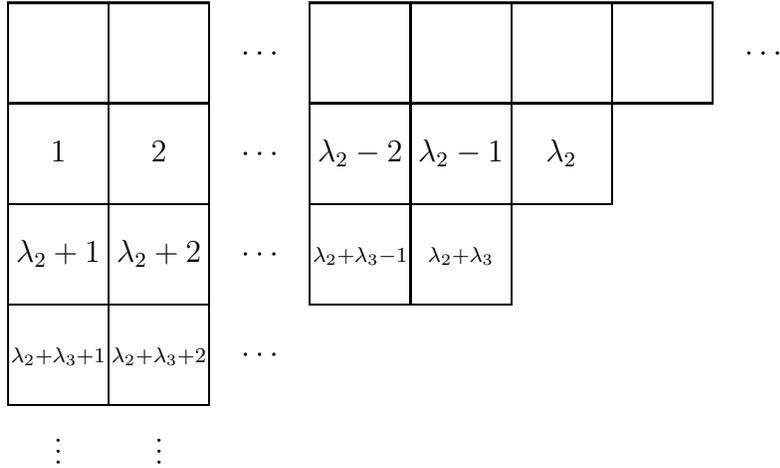
\ytableausetup
{mathmode, boxsize=3.2em}
\begin{ytableau}
{}& & \none[\cdots]&  & & & & \none[\cdots]   \\
1 & 2  & \none[\cdots] & \lambda_2 -2 & \lambda_2-1 & \lambda_2 \\
\lambda_2+1 & \lambda_2+2 & \none[\cdots] & \scriptstyle  \lambda_2+\lambda_3-1 & \scriptstyle  \lambda_2+\lambda_3 \\
\scriptstyle  \lambda_2+\lambda_3+1 & \scriptstyle  \lambda_2+\lambda_3+2 & \none[\cdots] \\
\none[\vdots]  & \none[\vdots] 
\end{ytableau} 
$$
Note that we have used $s:= \lambda_2 + \cdots + \lambda_l$ colors to paint  the boxes in the second  line to the last line. Note that $s = n-\lambda_1 <d$. Finally, we paint the boxes in the first line as follows. 
$$
\ytableausetup
{mathmode, boxsize=2.0em}
\begin{ytableau}
\scriptstyle s+1 &  \scriptstyle s+2 & \none[\cdots] &d&1&1& \none[\cdots] &1 & 2 & 2 & \none[\cdots] & 2 & 3&\none[\cdots]
\end{ytableau} 
$$ 
After the position   
$
\ytableausetup
{mathmode, boxsize=1.3em}
\begin{ytableau}
1&1& \none[\cdots] 
\end{ytableau} 
$
, each color $i$ appears $c_i-1$ times. Now it is easy to see that this coloring satisfies the expected condition. 

By Lemma~\ref{partition}, the claim we have just shown implies that $\height(\ISp_\lambda) \ge \lambda_1$. 
On the other hand，we know that  $P_F \supset \ISp_{\lambda}$ for $F \subset [n]$ with $\# F = \lambda_1+1$. 
Since $\height(P_F) =\lambda_1$,  $P_F$ is a minimal prime of $\ISp_\lambda$, and $\height(\ISp_\lambda)=\lambda_1$. 
\end{proof}
.

For an integer $k$ with  $2 \le k \le n$, set 
$$I_{n,k} := \bigcap_{\substack{F \subset [n], \\ \#F =k}} P_F.$$
Clearly, $\height(I_{n,k}) =k-1$.

\begin{prop}\label{minimal primes}
Let $\lambda=(\lambda_1, \ldots, \lambda_l)$ be a partition of $n$. 
We always have 
$$\sqrt{\ISp_\lambda} \subset I_{n, \lambda_1 +1},$$
and the equality holds if and only if $\lambda_{l-1}=\lambda_1$. 
\end{prop}

\begin{proof}
The former assertion immediately follows from Proposition~\ref{lambda_1+1}.  To prove the latter assertion,  assume that $\lambda_{l-1} < \lambda_1$. Set $m: = \max\{ i \mid \lambda_i = \lambda_1 \}$. Note that $m < l-1$ now. 
Take $\ba \in K^n$ so that $a_i=a_j$ if and only if $(k-1) \lambda_1 +	1 \le i,j \le k  \lambda_1$ for some $1 \le k \le m$, or $m \lambda_1 +1 \le i,j \le m \lambda_1 + \lambda_{m+1}+1$. 
Then it is easy to see that $\ba$ satisfies the condition $(*)$ of Lemma~\ref{0 point}, and hence $\ba \in V(\ISp_\lambda)$. On the other hand, 
since there is no $F \subset [n]$ with $\# F = \lambda_1+1$ such that $a_i =a_j$ for all $i,j \in F$, we have 
$\ba \not \in V( I_{n, \lambda_1 +1})$. This means that if we set 
$$F_1: =\{ 1,2, \ldots, \lambda_1 \}, \quad F_2: =\{ \lambda_1 +1,  \lambda_1+ 2, \ldots, 2\lambda_1 \}, \ldots,$$ 
$$ F_m: =\{ (m-1)\lambda_1 +1, \ldots,  m\lambda_1 \}, \, F_{m+1}: =\{ \, m\lambda_1 +1,  m\lambda_1+ 2, \ldots, m\lambda_1 +\lambda_{m+1}+1  \, \}, $$
and
$$P:=(x_i-x_j \mid  i,j \in F_k \ \text{for $k=1, \ldots, m+1$})$$
then we have $P \in V(\ISp_\lambda)$ but  $P \not \in V(I_{n, \lambda_1+1})$. Hence $\sqrt{\ISp_\lambda} \ne I_{n, \lambda_1 +1}.$

Next,  assume that $\lambda_{l-1}=\lambda_1$.  To show $\sqrt{\ISp_\lambda} = I_{n, \lambda_1 +1}$ (equivalently, $\sqrt{\ISp_\lambda} \supset I_{n, \lambda_1 +1}$ now), it suffices to show that $\ba \not \in V(I_{n, \lambda_1+1})$ implies $\ba \not \in V(\ISp_\lambda)$. Set $t:=\# \{ a_1, \ldots, a_n \}$. 
By the symmetry, we may assume that the indices are ``sorted'' as follows; these are integers 
$0=m_1 < m_2 < \cdots < m_t=n$ such that $a_i =a_j$ if and only if $m_k +1 \le i,j \le m_{k+1}$ for some $1 \le k < t$.  By the assumption, the cardinality of each block (i.e., $m_{k+1}-m_k$ for each $k$) is less than of equal to $\lambda_1$.  
Here, considering the following  tableau 
$$
\ytableausetup
{mathmode, boxsize=2.3em}
\begin{ytableau}
1 & 2  & 3& \none[\dots]
& \scriptstyle  \lambda_1 \\
\scriptstyle  \lambda_1+1  & \scriptstyle  \lambda_1+2 & \scriptstyle  \lambda_1+3 & \none[\dots]
& \scriptstyle  2 \lambda_1 \\
\scriptstyle  2\lambda_1+1  & \scriptstyle  2\lambda_1+2 & \scriptstyle  2\lambda_1+3 & \none[\dots]
& \scriptstyle  3 \lambda_1 \\
\none[\vdots] & \none[\vdots]
& \none[\vdots] & \none & \none[\vdots]\\
\end{ytableau}
$$
of shape $\lambda$, we see that  $\ba \not \in V(\ISp_\lambda)$.  
In fact, for any distinct $i,j$ in the same column of the above tableau, we have $a_i \ne a_j$. 
So $\ba$ does not satisfy the condition $(*)$, and hence $\ba \not \in V(\ISp_\lambda)$.   
\end{proof}

\begin{cor}\label{ISp  non pure}
Let $\lambda=(\lambda_1, \ldots, \lambda_l)$ be a partition of $n$. 
Then $\sqrt{\ISp_\lambda}$ is pure dimensional (i.e., all minimal primes of $\ISp_\lambda$ have the same height) if and only if 
$\lambda_{l-1}=\lambda_1$ or $\lambda_2 = 1$. 
\end{cor}

\begin{proof}
If $\lambda_{l-1}=\lambda_1$, then the assertion immediately follows from Proposition~\ref{minimal primes}. If $\lambda_2 =1$, then $\ISp_\lambda$ is Vandermonde type, 
and $R/\ISp_\lambda$ is Cohen--Macaulay as shown in \cite{WY}. 
So $\sqrt{\ISp_\lambda}$ (actually, $\ISp_\lambda$ itself) is pure dimensional. 

Conversely, if $\lambda_{l-1} <\lambda_1$, then the prime ideal $P$ introduced in the proof of Proposition~\ref{minimal primes} is a minimal prime of $\ISp_\lambda$. 
In fact, for a  prime ideal $P' \subsetneq P$, some point $\ba \in V(P')$ loses an equation $a_i =a_j$ which any point in $V(P)$ has. Then it is easy to see that $\ba$ does not satisfy the condition $(*)$ of Lemma~\ref{0 point}, and hence $\ba \not \in V(\ISp_\lambda)$. It means that $P' \not \supset \ISp_\lambda$.  

On the other hand, we have $\height(P)=m(\lambda_1-1)+\lambda_{m+1}$. 
It is easy to see that $\lambda_2 > 1$ implies $\height(P) > \lambda_1 \, (= \height(\ISp_\lambda))$, and $\ISp_\lambda$ is not pure. 
\end{proof}

By motivation from the study of rational Cherednik algebras,  
Etingof, Gorsky  and Losev \cite{EGL} proved the following result. 

\begin{thm}[{\cite[Proposition~3.11]{EGL}, see also \cite[Theorem~2.1]{BCES}}]\label{EGL} 
Suppose that $\chara(K)=0$. Then $R/I_{n,k}$ is Cohen--Macaulay if and only if $k=2$ or  $2k >n$. 
\end{thm}

The only if part of the above theorem can be slightly improved as follows.

\begin{prop}\label{nonCM}
If $k \ge 3$ and $2k \le n$, then $R/I_{n,k}$ does not satisfy Serre's condition $(S_2)$. 
\end{prop} 

\begin{proof}
Set $F=\{ 1,2, \ldots, k\}$ and $F'=\{k+1, k+2, \ldots, 2k\}$. 
Then $P_F=(x_1-x_2, \ldots, x_1-x_k)$ and $P_{F'} =(x_{k+1}-x_{k+2}, \ldots, x_{k+1}-x_{2k})$  are minimal primes of $I_{n,k}$.  Consider the prime ideal $P:= P_F +P_{F'}$. 
It is easy to see that any minimal prime of $I_{n,k}$ other than  $P_F$ and $P_{F'}$ is  not contained in $P$, that is, $\Min_{R_P}( (R/I_{n,k})_P) =\{   (P_F)_P, (P_{F'})_P \}$. 
By the Mayer–Vietoris sequence 
$$0 \too (R/I_{n,k})_P \too (R/P_F)_P \oplus (R/P_{F'})_P \too (R/P)_P \too 0$$
 we see that  $\depth (R/I_{n,k})_P=1$. On the other hand, 
$$\dim  (R/I_{n,k})_P = 
\height (P) -\height (I_{n,k}) = 2(k-1)-(k-1)=k-1 \ge 2.$$ Hence $R/I_{n,k}$ does not satisfy Serre's condition $(S_2)$. 
\end{proof}

\begin{prop}\label{CM ISP} 
Let $(\lambda_1, \ldots, \lambda_l)$ be a partition of $n$. 
If $R/\ISp_\lambda$ is Cohen--Macaulay, then one of the following conditions is satisfied,. 
\begin{itemize}
\item[(1)] $\lambda_2=1$, 
\item[(2)] $l=2$,
\item[(3)] $l=3$, $\lambda_1=\lambda_2$ and  $\lambda_3=1$. 
\end{itemize}
\end{prop}

\begin{proof}
If neither $\lambda_2 =1$ nor $\lambda_{l-1} = \lambda_1$ then $R/\ISp_\lambda$ is not Cohen--Macaulay by Corollary~\ref{ISp  non pure}. 
So we may assume that $\lambda_{l-1} = \lambda_1$ 
(note that this condition is clearly satisfied if $l=2$). 
In this case, we have $\sqrt{\ISp_\lambda} = I_{n, \lambda_1 +1}$ by Proposition~\ref{minimal primes}. 
If none of the conditions (1)--(3) is satisfied, then $\lambda_1 +1 \ge 3$ and 
$n \ge 2(\lambda_1 +1)$, and we can show that  $R/\ISp_\lambda$ does not satisfy $(S_2)$ by an argument similar to the proof of Proposition~\ref{nonCM}.  
In fact, in this situation, there is a prime ideal $P$ of $A:=(R/\ISp_\lambda)_P$ such that $\height P\ge 2$ and the ideal $(0) \subset A$ has a primary decomposition 
such that $(0)= \fq \cap \fq'$ and $\sqrt{\fq + \fq' }$ is the maximal ideal of $A$. It means that $\depth A =1$, while $\dim A \ge 2$. 
\end{proof}

In  the case (1) of Theorem~\ref {CM ISP}, $R/\ISp_\lambda$ is Cohen--Macaulay for arbitrary $K$ by \cite{WY}. 
So it remains to consider the cases (2) and (3). 
If $\ISp_\lambda$ is radical, we can use Theorem~\ref{EGL} in the case $\chara(K)=0$.
So the next problem is very natural.

\begin{conj}\label{radical?}
The Specht ideal  $\ISp_\lambda$ is always a radical ideal. 
\end{conj}

In the present paper, we will prove this conjecture in two important cases. 
We treat the case $\lambda=(n-d,d)$ in this section, and the case $\lambda=(a,a,1)$ in the next section. Clearly, the former (resp. latter) case corresponds to the condition (2) (resp. 3) of Proposition~\ref{CM ISP}.  
The following observation is useful in our proof. 

\medskip

If we set $y_i := x_i -x_n$ for $1 \le i \le n-1$, then we have $R=K[x_1, \ldots, x_n]=K[y_1 \ldots, y_{n-1}, x_n]$.  
Since $x_i-x_j =y_i-y_j$ for $1 \le i,j <n$, the Specht polynomial $f_T$ is a polynomial of $y_1, \ldots, y_{n-1}$ for all $T$.  Hence $x_n$ is an $R/\ISp_\lambda$-regular element.

Let $\lambda =(\lambda_1, \ldots, \lambda_l)$ be a partition of $n$ with $\lambda_{l-1}=\lambda_1$, and $\mu$ the partition of $n-1$ given by
$$
\mu =
\begin{cases}
(\lambda_1, \ldots, \lambda_{l-1}, \lambda_l-1) & \text{(if $\lambda_l \ge 2$)}\\
(\lambda_1, \ldots, \lambda_{l-1}) & \text{(if $\lambda_l = 1$).}
\end{cases}
$$

\begin{lem}\label{radical of frJ}
Let $\lambda$ and $\mu$ be as above, $\ISp_\mu \subset S:=K[x_1, \ldots, x_{n-1}]$ the Specht ideal of $\mu$, and $I_{\< m \>} \subset S$ the ideal generated by all degree $m$ squarefree monomials. And let $\varphi: R \to S \, (\cong R/(x_n))$ be the natural surjection. 
Then we have 
$$\sqrt{\varphi(\ISp_\lambda)} = \varphi(I_{n, \lambda_1+1}) = \sqrt{\ISp_\mu} \cap I_{\<n-\lambda_1\>}.$$
\end{lem}

\begin{proof}
By  Proposition~\ref{minimal primes}, we have $\sqrt{\ISp_\lambda} = I_{n, \lambda_1 +1}$. 
Since $\mu$ satisfies the condition of Proposition~\ref{minimal primes} again, if we set  $\bP_F =(x_i -x_j \mid i, j \in F)\subset S$ for  $F \subset [n-1]$, then 
$$\sqrt{\ISp_\mu}  = \bigcap_{\substack{F \subset [n-1], \\ \#F =\lambda_1+1}} \bP_F.$$ 
If  $n \not \in F$ (i.e., $F \subset [n-1]$), then we have $\varphi(P_F) = \bP_F$. If $n \in F$, then $\varphi(P_F)= (x_i \mid i < n, i \in F)$.  

For $F, F' \subset [n]$, consider 
$$0 \too R/(P_F \cap P_{F'}) \too R/P_F \oplus  R/P_{F'} \too R/(P_F + P_{F'}) \too 0.$$
Since $P_F + P_{F'}$ is generated by polynomials of $y_i \, (= x_i -x_n)$ for 
$1 \le i \le n-1$, $x_n$ is a non-zero divisor over  $R/(P_F + P_{F'})$. So applying $-\otimes_R S$, we have the exact sequence 
$$0 \too S/\varphi(P_F \cap P_{F'}) \too S/\varphi(P_F) \oplus  S/\varphi(P_{F'}) \too S/\varphi(P_F + P_{F'}) \too 0$$
by \cite[Proposition~1.1.4]{BH}.  
It means that $\varphi(P_F \cap P_{F'}) = \varphi(P_F) \cap \varphi(P_{F'})$. Repeating this argument,  we have 
$$\varphi(I_{n,\lambda_1 +1}) =  \bigcap_{\substack{F \subset [n], \\ \#F =\lambda_1 +1}} \varphi(P_F),$$
and hence $\varphi(I_{n,\lambda_1 +1})$ is radical. Now we have  
\begin{eqnarray*}
\sqrt{\varphi(\ISp_\lambda}) = \sqrt{\varphi(I_{n,\lambda_1 +1})} 
&=& \varphi(I_{n,\lambda_1 +1}) \\
&=& \bigcap_{\substack{F \subset [n], \\ \#F =\lambda_1 +1}} \varphi(P_F)\\
&=&  \Biggl( \bigcap_{\substack{F \subset [n], \, n \not \in F \\ \#F =\lambda_1 +1}} \varphi(P_F) \Biggr)  \cap \Biggl( \bigcap_{\substack{F \subset [n], \, n  \in F \\ \#F =\lambda_1 +1}} \varphi(P_F) \Biggr) \\
&=&  \Biggl( \bigcap_{\substack{F \subset [n-1], \\ \#F =\lambda_1 +1}} \bP_F \Biggr)  \cap \Biggl( \bigcap_{\substack{F \subset [n-1], \\ \#F =\lambda_1}} (x_i \mid i \in F) \Biggr) \\
&=& \sqrt{\ISp_\mu} \cap I_{\<n-\lambda_1 \>}. 
\end{eqnarray*}
\end{proof}

\section{The radicalness of $\ISp_{(n-d,d)}$}
In this section, we will prove Conjecture~\ref{radical?} in the case $\lambda=(n-d,d)$.  

\begin{thm}\label{ISp_{(n-k,k)} is radical}
The Specht ideal $\ISp_{(n-d,d)}$ is radical. 
\end{thm}

We prove this theorem by induction on $d$  (since $R/\ISp_{(n-1,1)} \cong K[X]$, the assertion is clear if $d=1$). 

For a polynomial $f  \in S=K[x_1, \ldots, x_{n-1}]$, let 
$\supp(f) $ be the set of squarefree monomials in $S$ which divide some nonzero term of $f$. 
\begin{ex}
$\supp(x_1x_2^3-3(x_2x_3)^2)=\{1, x_1, x_2, x_3, x_1x_2, x_2x_3 \}$. 
\end{ex}

Let 
$$
T = 
\ytableausetup
{mathmode, boxsize=2em}
\begin{ytableau}
i_1 & i_2  & \cdots &i_d & i_{d+1} & \cdots & i_{n-d}  \\
j_1 & j_2 & \cdots & j_d
\end{ytableau}
$$
be a tableau of shape $\lambda$. The Specht polynomial $f_T$ does not care the order of the 1st to the $d$-th columns, and the order of the $(d+1)$-st to the $(n-d)$-th columns. Moreover, if we permute $i_k$ and $j_k$ for some $1 \le k \le d$ then its Specht polynomial becomes $-f_T$. On the set of tableaux of shape $\lambda$, we consider the equivalence relation modulo these permutations. Then $T \equiv T'$ if and only if $f_T =\pm f_{T'}$.  
Let $\Tab(\lambda)$ denote the set of the equivalence classes. However, we sometimes identify an equivalence class $[T] \in \Tab(\lambda)$ with its representative $T$.  For example,  we often write like $T \in \Tab(\lambda)$. When we consider $f_T$ of $[T] \in \Tab(\lambda)$, we assume that $i_k < j_k$ for all  $1 \le k \le d$ unless otherwise specified. Clearly, 
$$\ISp_\lambda = (\, f_T \mid T \in \Tab(\lambda) \, ).$$

Let $\StTab(\lambda)$ denote the set of standard tableaux of shape $\lambda$. Note that an equivalence class $[T] \in \Tab(\lambda)$ contains {\it at most} one standard tableau. If $[T]$ contains a standard tableau, we say it is standard.

\begin{lem}\label{varphi}
Recall that $S=K[x_1, \ldots, x_{n-1}]$. 
Consider the partitions  $\lambda = (n-d, d)$ and  $\mu = (n-d, d-1)$. For the natural surjection 
$\varphi: R \to S$, we have  
 $$\varphi(\ISp_\lambda )  = ( \, x_i f_T \mid T \in \Tab(\mu), x_i \not \in \supp(f_\mu) \, ) \subset S.$$
\end{lem}

For notational simplicity,  we set 
$$\frJ :=   ( \, x_i f_T \mid T \in \Tab(\mu), x_i \not \in \supp(f_\mu) \, ).$$

\begin{proof}
It is well-known that 
$$\ISp_\lambda =  ( \, f_{T'} \mid  T' \in \StTab(\lambda) \, ).$$
However, we use the inverse order $n \prec n-1 \prec \cdots  \prec 1 $ here. Hence a standard tableau on $\lambda$  is of the form 
$$
T' = 
\ytableausetup
{mathmode, boxsize=2em}
\begin{ytableau}
n & i_2  & i_3 & \cdots & i_d & \cdots & i_{n-d}  \\
j_1 & j_2 & j_3 & \cdots & j_d   
\end{ytableau}
$$
and we have $$\varphi(f_{T'}) = -x_{j_1} \prod_{k=2}^d (x_{i_k}-x_{j_k}).$$ 
For the tableau 
$$
T= \ytableausetup
{mathmode, boxsize=2em}
\begin{ytableau}
 i_2  & i_3 & \cdots &i_d &  \cdots & i_{n-d}  & j_1\\
j_2 & j_3 & \cdots & j_d  
\end{ytableau}
$$
of shape $\mu$, we have $\varphi(f_{T'}) = -x_{j_1} f_T$ and $x_{j_1} \not \in \supp(f_T)$. 
Hence we have $\varphi(\ISp_\lambda) \subset \frJ$. 

The converse inclusion follows  from a similar argument.   
\end{proof}

The following fact (and its local analog) must be well-known, but we give a proof for the reader's convenience.   

\begin{lem}\label{radical reduction}
Let $A=\bigoplus_{i \in \NN} A_i $ be a noetherian graded ring, and $a \in A$ a homogeneous non-zero divisor of positive degree. If $A/aA$ is reduced, then $A$ is also.  
\end{lem}

\begin{proof}
Since $aA$ is a radical ideal, there are prime ideals $P_1, \ldots, P_m$ of $A$ such that $\sqrt{I} =\bigcap_{i=1}^m P_i$. Since $a$ is a non-zero divisor, we have $\height(P_i) \ge 1$ for all $i$.  
So, for each $i$, we can take a minimal prime $Q_i$ of $A$ contained in $P_i$. Take a homogeneous element $x \in \bigcap_{i=1}^m Q_i$. Since $\bigcap_{i=1}^m Q_i \subset \bigcap_{i=1}^m P_i =aA$,  
there is a homogeneous element $x_1 \in A$ such that $x=a x_1$. Since $a$ is a non-zero divisor, we have $a \not \in Q_i$ for all $i$, and it means that $x_1 \in Q_i$ for all $i$, and hence $x_1 \in \bigcap_{i=1}^m Q_i$.  Applying the above argument to $x_1$, we can find $x_2 \in A$ such that $x_1 = a x_2$, that is, $x= a^2 x_2$. Repeating this argument, we have $x  \in \bigcap_{i=1}^\infty a^i A = (0)$, and it implies that $\bigcap_{i=1}^m Q_i =(0)$. Since  $Q_i$'s are prime ideals, $(0)$ is a radical ideal. 
\end{proof}

For partitions $\lambda=(n-d,d)$ and $\mu=(n-d,d-1)$, we assume the induction hypothesis that $\ISp_\mu$ is radical. By Lemma \ref{radical of frJ}, we have   $\sqrt{\frJ} = \ISp_\mu \cap I_{\<d \>}$. 
If  $\frJ =\varphi(\ISp_\lambda)$ is radical, then so is $\ISp_\lambda$ itself  by Lemma~\ref{radical reduction}.
So it suffices to show that $\frJ \supset \ISp_\mu \cap I_{\<d\>}. $

\begin{lem}\label{reduction}
Let $\lambda$ and $\mu$ be as above.  Assume that $\ISp_\mu$ is a radical ideal. Then $\ISp_\lambda$ is a radical ideal, if the following condition is satisfied. 
\begin{itemize}
\item[$(*)$] If  $\phi = x^\ba(\sum_{T \in \Tab(\mu)} c_T f_T)  \in I_{\<d\>}$ for some squareferee monomial $x^\ba \in S=K[x_1, \ldots, x_{n-1}]$ and  $c_T \in K$, 
we have $\phi \in   \frJ$.  
\end{itemize}
\end{lem}

\begin{proof}
First,  note that the assumption that $x^\ba$ is squarefree can be easily dropped. In fact,  for any $\bb \in \NN^{n-1}$ and $h \in \ISp_\mu$, 
$x^\bb h  \in I_{\<d\>}$ implies   $(\prod_{b_i > 0} x_i )\cdot h \in I_{\<d\>}$, and  $(\prod_{b_i > 0} x_i )\cdot h \in \frJ$ implies $x^\bb h  \in \frJ$. 

By the remark just before the lemma, $\ISp_\lambda$ is a radical ideal, if the condition 
\begin{itemize}
\item[$(**)$] If  $\psi = \sum_{T \in \Tab(\mu)} g_T f_T  \in I_{\<d\>}$ for some polynomials $g_T \in S$, then $\psi \in   \frJ$. 
\end{itemize}  is satisfied. 
So it suffices to show that $(*)$ implies $(**)$. 

Assume that $\psi = \sum_{T \in \Tab(\mu)} g_T f_T  \in I_{\<d\>}$.   
Take $\ba \in \NN^{n-1}$, and let $c_T x^\ba$ be the degree $\ba$ term of $g_T$    
(of course, $c_T$ can be 0). Now we want to show that  
$$\psi_\ba:=  x^\ba \sum_{T \in \Tab(\mu)} c_T f_T$$ 
is contained in $I_{\<d\>}$. 
By contradiction, we assume that the degree  $\bb$ term of $\psi_\ba$ is not contained in $I_{\<d\>}$.  
Since  $x^\bb \not \in I_{\<a\>}$, and all terms of $f_T$ are squarefree and have degree $d-1$, we have $x^\ba = x^\bb/({\prod_{b_i > 0}x_i)}$. Hence the degree $\bb$ term of $g_Tf_T$ equals that of $c_T x^\ba f_T$. Therefore, the degree $\bb$ term of   $\psi_\ba$ coincides with that of  $\psi \in I_{\<d\>}$. 
This is a contradiction, and hence we have $\psi_\ba \in I_{\<d\>}$. Since $\psi_\ba$ can  play the role of $\phi$ in  the condition $(*)$
, so $(*)$ implies $\psi_\ba \in \frJ$. Hence 
$\psi = \sum_{\ba \in \NN^{n-1}} \psi_\ba \in  \frJ$, and $(*)$ implies $(**)$ 
\end{proof}

\begin{lem}\label{support}
With the same notation as Lemma~\ref{reduction}, 
if  $x^\ba \not \in \supp(f_T)$ for a tableau $T$ of shape $\mu$, then we have  $x^\ba f_T \in  \frJ.$  
\end{lem}

\begin{proof}
We may assume that if $x_i$ divides $x^\ba$ then it belongs to $\supp(f_T)$. 
Then, by the shape of the Specht polynomial $f_T$, there are distinct $i,j \in [n-1]$ such that $x_ix_j \, | \, x^\ba$ but $x_i x_j \not \in \supp(f_T)$. 
Now $i$ and $j$ are in the same column in $T$, and we may assume that $T$ is of the form 
$$
\ytableausetup
{mathmode, boxsize=2em}
\begin{ytableau}
i & i_2  & \cdots &i_{d-1} & k  & \none[\cdots]\\
j & j_2 & \cdots & j_{d-1}  \\
\end{ytableau}
$$
(note that $n-d > d-1$ now). 
Consider two other tableaux on $\mu$ as follows 
$$
T_1 =\ytableausetup
{mathmode, boxsize=2em}
\begin{ytableau}
i & i_2  & \cdots &i_{d-1} & j  & \none[\cdots]\\
k & j_2 & \cdots & j_{d-1}  \\
\end{ytableau} 
\quad \text{and} \quad 
T_2 =\ytableausetup
{mathmode, boxsize=2em}
\begin{ytableau}
k & i_2  & \cdots &i_{d-1} & i  & \none[\cdots]\\
j & j_2 & \cdots & j_{d-1}  \\
\end{ytableau} 
$$
Then we have $f_T =f_{T_1} +f_{T_2}$. Hence 
$$x_ix_j f_T = x_ix_j (f_{T_1} +f_{T_2}) = x_i (x_j f_{T_1}) + x_j (x_i f_{T_2}).$$
Since $x_j  \not\in \supp(f_{T_1})$, we have $x_j f_{T_1} \in \frJ$. 
Similarly, $x_i f_{T_2} \in \frJ$. 
Hence $x_ix_j f_T \in  \frJ$, and it implies that $x^\ba f_T \in \frJ$.
\end{proof}

In the  condition $(*)$ of Lemma~\ref{reduction}, we may assume that $c_T \ne 0$ implies $x^\ba \in \supp(f_T)$ by Lemma~\ref{support}, and $x^\ba = x_1x_2\cdots x_k$ by the symmetry. Hence we have the following.  

\begin{cor}\label{reduction2}
Let $\lambda$ and $\mu$ be as above. Assume that $\ISp_\mu$ is a radical ideal. Then $\ISp_\lambda$ is a radical ideal, if the following condition is satisfied. 
\begin{itemize}
\item[$(*\!*\!*)$] 
For  the squarefree monomial $x^\ba = x_1x_2\cdots x_k \in S$  with $1 \le k \le d-1$, set $X:= \{\, T \in \Tab(\mu) \mid x^\ba \in \supp(f_T) \, \}$. If  $\phi = x^\ba(\sum_{T \in X} c_T f_T)  \in I_{\<d\>}$ for some $c_T \in K$, we have $\phi \in   \frJ.$ 
\end{itemize}
\end{cor}

In the sequel, $X$ means the set defined in $(*\!*\!*)$. 
An element of $X$ has the following ``normal form" 
\begin{equation}\label{T}
T= 
\ytableausetup
{mathmode, boxsize=2.3em}
\begin{ytableau}
1 & 2  & \cdots & k & i_{k+1} & \cdots & i_{d-1} & i_d&  i_{d+1} & \cdots & i_{n-d} \\
j_1 & j_2 & \cdots &j_k & j_{k+1} & \cdots &j_{d-1}  
\end{ytableau}, 
\end{equation}
where $i_{k+1} < i_{k+2} < \cdots < i_{d-1}$, $i_d < i_{d+1} < \cdots < i_{n-d}$ 
and $i_l < j_l$ for all $k < l < d$. 

\begin{lem}\label{permutation}
With  the same notation as Corollary~\ref{reduction2},   let $T \in X$ be a tableau of the form \eqref{T}.  For any permutation $\sigma$ on $\{i_d, \ldots, i_{n-d}, j_1, \ldots, j_k \}$, we have $x^\ba(f_T - f_{\sigma T}) \in \frJ$. Here $\sigma T$ is the Young tableau of shape $\mu$ given by replacing each $i$ in $T$ by $\sigma(i)$.   
\end{lem}

\begin{proof}
It is easy to see that $\sigma$ is a product of transpositions of the form $\tau= (i_a, j_b)$ for $d \le a \le n-d$ and $1 \le b \le k$. 
So it suffices to show that $x^\ba(f_T - f_{\tau T}) \in \frJ$.   
By the symmetry,  we may assume that $\tau = (i_d, j_1)$.  
We have 
$$
\tau T= 
\ytableausetup
{mathmode, boxsize=2.3em}
\begin{ytableau}
1 & 2  & \cdots & k & i_{k+1} & \cdots & i_{d-1} & j_1 & i_{d+1}&\cdots & i_{n-d} \\
i_d & j_2 & \cdots &j_k & j_{k+1} & \cdots &j_{d-1}   
\end{ytableau} 
$$
Set 
$$
T'= 
\ytableausetup
{mathmode, boxsize=2.3em}
\begin{ytableau}
i_d & 2  & \cdots & k & i_{k+1} & \cdots & i_{d-1} & 1 &  i_{d+1} & \cdots & i_{n-d} \\
j_1 & j_2 & \cdots &j_k & j_{k+1} & \cdots &j_{d-1}   
\end{ytableau} 
$$
Then $x^\ba(f_T -f_{\tau T}) = x^\ba f_{T'} \in  \frJ,$
since $x_1 \not \in \supp(f_{T'}).$
\end{proof} 

For the tableau $T$ of \eqref{T}, set 
$$h_T := \prod_{l=k+1}^{d-1}(x_{i_l}-x_{j_l}).$$ 

\begin{lem}\label{relation of h_T}
With the above notation, if $\phi = x^\ba(\sum_{T \in X} c_T f_T)  \in I_{\<d\>}$ for some $c_T \in K$, then we have $$\sum_{T \in X} c_T h_T =0.$$ 
\end{lem}

\begin{proof}
Since  $x^\ba(f_T -x^{\ba}h_T) \in I_{\<d\>}$ for each $T \in X$, we have 
$$\phi - x^{2\ba} \sum_{T \in X} c_T h_T = x^\ba \sum_{T \in X} c_T (f_T -x^{\ba}h_T  )\in I_{\<d\>},$$ and hence $ x^{2\ba} (\sum_{T \in X} c_T h_T) \in I_{\<d\>}$.  
On the other hand, any nonzero term of $x^{2\ba} (\sum_{T \in X} c_T h_T)$ dose not belong to $I_{\<d\>}$. 
Hence we have $\sum_{T \in X} c_T h_T =0$. 
\end{proof}


\medskip

For the tableau $T$  of \eqref{T}, consider the tableau 
\begin{equation}
\widehat{T}= 
\ytableausetup
{mathmode, boxsize=2.3em}
\begin{ytableau}
i_{k+1} & \cdots & i_{d-1} & i_d&  i_{d+1} & \cdots & i_{n-d} \\
j_{k+1} & \cdots &j_{d-1}   
\end{ytableau} 
\end{equation}
of shape $(n-d-k, d-1-k)$, 
and the tableaux 
\begin{equation}
\widetilde{T}= 
\ytableausetup
{mathmode, boxsize=2.3em}
\begin{ytableau}
i_{k+1} & \cdots & i_{d-1} & i_d&  i_{d+1} & \cdots & i_{n-d}  & j_1 & j_2 & \cdots &j_k \\
j_{k+1} & \cdots &j_{d-1}   
\end{ytableau} 
\end{equation}
on shape $(n-d, d-1-k)$. Clearly, $h_T = f_{\widehat{T}} = f_{\widetilde{T}}$. 

We say $T \in X$ is {\it quasi $h$-standard} (resp. {\it $h$-standard}), if $\widehat{T}$ (resp.  $\widetilde{T}$ ) is a standard tableau. Here we regard $\widehat{T}$ and  $\widetilde{T}$ as the tableaux with the letter set $[n] \setminus ([k] \cup \{j_1, \ldots, j_k\})$ and $[n] \setminus [k]$ respectively.  
Set $$Y:=\{ T \in X \mid \text{$T$ is quasi $h$-standard} \}$$
and 
$$Z:=\{ T \in X \mid \text{$T$ is $h$-standard} \}.$$

\begin{lem}\label{quasi h-standard}
If $T \in X$, there are $T_1, \ldots, T_m  \in Y$ and   $c_1, \ldots, c_m \in K$ for such that 
$$f_T =  \sum_{l=1}^m c_l f_{T_l}.$$ 
Moreover, if $T$ is of the form \eqref{T}, then we may assume that each $T_l$ is of the form
$$
\ytableausetup
{mathmode, boxsize=2.3em}
\begin{ytableau}
1 & 2  & \cdots & k & i_{k+1}' & \cdots & i_{d-1}' & i_d'&  i_{d+1}' & \cdots & i_{n-d}' \\
j_1 & j_2 & \cdots &j_k & j_{k+1}' & \cdots &j_{d-1}'   
\end{ytableau} 
$$
with $(i_d, i_{d+1}, \ldots, i_{n-d} )\le (i_d', i_{d+1}', \ldots, i_{n-d}' )$ in the coordinate-wise order.
\end{lem}

\begin{proof}
It suffices to apply the standard  argument of the Specht module theory (see, for example, \cite[\S 2.6]{Sa}) to the tableaux of the form $\widehat{T}$. For the reader's convenience, we will sketch the outline.  

Note that the tableau  $T$  in \eqref{T} belongs to $Y$ if and only if $j_{k+1} < j_{k+2} < \cdots < j_{d-1}$ and $i_{d-1} < i_d$.  
First, assume that $j_{l-1} > j_l$ for some $k+2 \le   l \le d-1$, and $l$ is the minimal number with this property 
(if there is no such $l$, let us move directly to the operation in the next paragraph).  Note that $i_{l-1} < i_l < j_l < j_{l-1}$ now. Consider the following  two tableaux
$$
T_a=
\begin{ytableau}
\cdots & i_{l-1} & i_l &  \cdots & \cdots \\
\cdots & j_l & j_{l-1} &  \cdots 
\end{ytableau} 
\qquad \text{and} \qquad 
T_b=
\begin{ytableau}
\cdots & i_{l-1} & j_l &  \cdots & \cdots \\
\cdots & i_l & j_{l-1} &  \cdots 
\end{ytableau}.
$$
More precisely, we have to apply a suitable column permutation  to $T_b$ so that the  first row is increasing from left to right.  Except the three slots where $j_{l-1}$, $i_l$ and  $j_l$ are in, $T_a$ and $T_b$ are same as $T$ (modulo the column permutation stated above).  Since $f_T = f_{T_a} -f_{T_b}$, we replace $f_T$ by $f_{T_a}-f_{T_b}$. If $T_a \not \in Y$ or $T_b \not \in Y$, we apply the above operation to them. Repeating this procedure, we can reduce to the case where 
$j_{k+1} < j_{k+2} < \cdots < j_{d-1}$ in \eqref{T}.

In the above situation, $i_{d-1} < i_d$ implies $T \in Y$. So we assume that $i_{d-1} > i_d$. Note that $i_d < i_{d-1} < j_{d-1}$ now. Consider the following  two tableaux
$$
T_c=
\begin{ytableau}
\cdots & i_d & i_{d-1} &  \cdots\\
\cdots & j_{d-1} 
\end{ytableau}
\qquad \text{and} \qquad 
T_d=
\begin{ytableau}
\cdots & i_d & j_{d-1} &  \cdots\\
\cdots & i_{d-1} 
\end{ytableau}
$$
(more precisely, we have to apply a suitable column permutation to each tableau so that the  first row is increasing from left to right). Since $f_T = f_{T_c} -f_{T_d}$, we replace $f_T$ by $f_{T_c}-f_{T_d}$. 
However, the above  column permutations of $T_c$ and $T_d$ might violate the inequalities $j_{k+1} < j_{k+2} < \cdots < j_{d-1}$. If this is the case,  we apply (and repeat, if necessary) ``$T_a$ and $T_b$ operations". After that, we go back to ``$T_c$ and $T_d$ operations". 
Repeating this procedure, we can get the expected representation $f_T =  \sum_{l=1}^m c_l f_{T_l}$. 

The last assertion of the lemma is clear, since $i_d < i_{d-1}$ in $T_c$ and  $i_d < j_{d-1}$ in $T_d$ (this fact also guarantees the termination of the above procedure). 
\end{proof}

\noindent{\it The proof of Theorem~\ref{ISp_{(n-k,k)} is radical}.}
Since $\ISp_\mu$ is a radical ideal by induction hypothesis, we can use Corollary~\ref{reduction2},  and it suffices to show the statement $(*\!*\!*)$. 
For a given  $\phi = x^\ba(\sum_{T \in X} c_T f_T)  \in I_{\<d\>}$, we apply the following algorithm.

\medskip

\noindent{\bf Operation 1.}  Using Lemma~\ref{quasi h-standard}, we re-write $\phi$ as $\phi= x^\ba(\sum_{T \in Y} c_T' f_T) $.

\medskip

Note that $$\sum_{T \in Y} c_T' f_{\widetilde{T}} = \sum_{T \in Y} c_T' h_T= 0$$
by Lemma~\ref{relation of h_T}. 
If $c'_T \ne 0$ implies $T \in Z$ (equivalently, $\phi= x^\ba(\sum_{T \in Z} c_T' f_T) $), then we have $\sum_{T \in Z} c_T' f_{\widetilde{T}} = 0$. 
Since $\{ f_{\widetilde{T}} \mid T \in Z \}$ is linearly independent as is well-known in the Specht module theory 
(see, for example, \cite[Theorem~2.5.2]{Sa}), we have $c'_T =0$ for all $T \in Z$ and hence $\phi=0 \in \frJ$.

If $c_T' \ne 0$ for some $T \in Y \setminus Z$, we go to the next operation.

\medskip

\noindent{\bf Operation 2.}  For the tableau $T$  of  the form \eqref{T}, we have a permutation $\sigma_T$ on $\{i_d, \ldots, i_{n-d}, j_1, \ldots, j_k \}$ such that 
$$\sigma_T(i_d) < \sigma_T(i_{d+1})  < \cdots < \sigma_T(i_{n-d}) < \sigma_T(j_1) <  \sigma_T(j_2) < \cdots < \sigma_T(j_k).$$
(Note that $T \in Z$ if and only if  $\sigma_T ={\rm Id}$.) 
Set 
$$\phi' =\sum_{T \in Y} c_T' f_{\sigma_T T}.$$

\bigskip

Note that 
$$\phi -\phi' = \sum_{T \in Y}c_T'(f_T - f_{\sigma_T T} ) \in \frJ.$$
by Lemma~\ref{permutation}. Hence it suffices to show that $\phi' \in \frJ$, and we replace $\phi$ by $\phi'$. 
However, $\sigma_T T \not \in Y$ (i.e., $i_{d-1} > i_d$) might happen, so we will go back to Operation~1, and then move to  Operation~2. We repeat this procedure until we get a form  $\phi= x^\ba(\sum_{T \in Z} c_T' f_T)$. 
 Operation~1 does not change the sequence $(j_1, \ldots, j_k)$, and Operation~2 raises it with respect to the coordinate-wise order.   It means that this algorithm eventually stops, and we get an expected expression  $\phi= x^\ba(\sum_{T \in Z} c_T' f_T)$. 
Then  $\phi=0$, as we have shown above. It means that the original $\phi$ can be reduced to 0 modulo $\frJ$, that is,  the original $\phi$ belongs $\frJ$. 
So  the condition  $(*\!*\!*)$ holds.  
\qed

\medskip

Combining Theorem~\ref{ISp_{(n-k,k)} is radical} with Theorem~\ref{EGL}, we have the following. 

\begin{cor}\label{ISp_{(n-k,k)} is CM}
If $\chara(K)=0$, the ring $R/\ISp_{(n-d,d)}$ is Cohen-Macaulay. 
\end{cor}

For $n \in \NN$, the $n$-th {\it Catalan number} $C_n$ is given by 
$$C_n = \frac{1}{2n+1}\binom{2n+1}{n},$$  
and $\{C_n\}_{n \in \NN}$ is one of the most important combinatorial sequences.   
Exercise A8 of the monograph \cite{St} gives 
17 algebraic interpretations of Catalan numbers  (e.g., in terms of the ring of upper triangular matrices, $SL(2,\CC)$, the toric variety associated with the $n$-dimensional cube, $\ldots$) beside more than 200 purely combinatorial interpretations. 
It is well-known that the number of standard tableaux of shape $(n,n)$ (or equivalently, of shape $(n,n-1)$) is $C_n$. See, for example, \cite[Exercise 168]{St} (this fact is counted as a combinatorial interpretation in this monograph). 
Theorem~\ref{ISp_{(n-k,k)} is radical}  gives yet another algebraic interpretation of the Catalan numbers. 

\begin{cor}
In the polynomial ring $K[x_1, \ldots, x_{2n}]$, the number of minimal generators of the ideal
$$I_{2n, n+1}= \bigcap_{\substack{F \subset [2n] \\ \# F = n+1}} (x_i-x_j \mid i, j \in F)$$ 
is the $n$-th Catalan number $C_n$.  Similarly, the ideal $I_{2n-1, n+1} \subset K[x_1, \ldots, x_{2n-1}]$ 
is also generated by $C_n$ elements.  
\end{cor}

\begin{proof}
Since $I_{2n, n+1} =\ISp_{(n,n)}$ and  $I_{2n-1, n+1} =\ISp_{(n,n-1)}$, the assertion follows from the above mentioned characterization of the Catalan numbers. 
\end{proof}

\section{The radicalness of $\ISp_{(a,a,1)}$}
In this section, we assume that  $n= 2a+1$ for some $a \in \NN$, and set $\lambda=(a,a,1)$ and $\mu =(a,a)$. 
For a Young tableau
\begin{equation}\label{(a,a,1)}
T = 
\ytableausetup
{mathmode, boxsize=2em}
\begin{ytableau}
i_1 & i_2  & \cdots &i_a  \\
j_1 & j_2 & \cdots & j_a
\end{ytableau}
\end{equation}
of shape $\mu$, set 
$$P(T):=\{  \, (i_k, j_k ) \mid 1\le k \le a  \, \}. $$

\begin{lem}\label{varphi2}
Recall that $S=K[x_1, \ldots, x_{n-1}]$, and $\varphi: R \to S$ is the natural surjection.  
With the above notation, we have  
 $$\varphi(\ISp_\lambda )  = ( \, x_i x_j f_T \mid T \in \Tab(\mu), (i, j)  \in P(T) \, ) \subset S.$$
\end{lem}

For notational simplicity,  we set 
$$\frJ := ( \, x_i x_j f_T \mid T \in \Tab(\mu), (i, j)  \in P(T) \, ).$$

\begin{proof}
The proof is parallel to that of Lemma~\ref{varphi}. 
Note that 
$$\ISp_\lambda =  ( \, f_{T'} \mid  T' \in \StTab(\lambda) \, ).$$
Here we use the inverse order $n \prec n-1 \prec \cdots  \prec 1 $, and hence a  standard tableau on $\lambda$ is of the form 
$$
T' = 
\ytableausetup
{mathmode, boxsize=2em}
\begin{ytableau}
n & i_2  & i_3 & \cdots &i_a \\
i_1 & j_2 & j_3 & \cdots & j_a   \\
j_1
\end{ytableau}
$$
and we have $$\varphi(f_{T'}) = x_{i_1}x_{j_1} \prod_{k=1}^{a} (x_{i_k}-x_{j_k})= x_{i_1}x_{j_1}f_T,$$
where $T$ is the tableau in \eqref{(a,a,1)}.  
It is easy to see that  $\varphi(f_{T'}) \in \frJ$, and hence $\varphi(\ISp_\lambda) \subset \frJ$. 

The converse inclusion is easy.   
\end{proof}

\begin{thm}\label{ISp_{(a,a,1)} is radical}
The Specht ideal $\ISp_{(a,a,1)}$ is radical.
\end{thm}

\begin{proof}
This can be proved by an argument similar to the previous section. Let $I_{\<a+1\>} \subset S$ be  the ideal 
generated by all squarefree  monomials of degree $a+1$. 
 With the above notation, we have  
$$\sqrt{\varphi(\ISp_\lambda)} = \sqrt{\ISp_\mu} \cap I_{\<a+1 \>}=\ISp_\mu \cap I_{\<a+1 \>},$$
where the first (resp. second) equality follows from Lemma~\ref{radical of frJ} (resp. Theorem~\ref{ISp_{(n-k,k)} is radical}). 
By Lemma~\ref{radical reduction}, it suffices to show that $\frJ \, (= \varphi(\ISp_\lambda))$ is a radical ideal, equivalently, $\frJ  \supset \ISp_\mu \cap I_{\<a+1 \>}$.  

By an argument similar to Lemma~\ref{reduction}, it suffices to show that 
\begin{itemize}
\item[$(\star)$] If  $\psi = x^\ba(\sum_{T \in \Tab(\mu)} c_T f_T)  \in I_{\< a+1\>}$ for some squareferee monomial $x^\ba \in S=K[x_1, \ldots, x_{n-1}]$ and  $c_T \in K$, 
we have $\psi \in   \frJ$.  
\end{itemize} 
By the symmetry, we may assume that $x^\ba =x_1x_2\cdots x_k$. We can rewrite $\psi$ as 
$\psi = x^\ba(\sum_{T \in \StTab(\mu)} c_T' f_T)$. Moreover, we can replace $\psi$ by 
$$
\psi': = x^\ba(\sum_{T \in W} c_T' f_T), 
$$
where $W$ is the subset of $\StTab(\mu)$ consisting of $T$ of the form  
\begin{equation}\label{(a,a)}
T= \ytableausetup
{mathmode, boxsize=2em}
\begin{ytableau}
 1 & 2  & \cdots &k & i_{k+1} & \cdots &i_a \\
j_1& j_2  & \cdots &j_k & j_{k+1} & \cdots &j_a \\
\end{ytableau}.
\end{equation}
In fact, if $T \in \StTab(\mu) -W$, then it is clear that $x^\ba f_T \in \frJ$. 

For the tableau $T$ in \eqref{(a,a)}, set 
$$h_T := \prod_{l=k+1}^a(x_{i_l}-x_{j_l}),$$ 
and consider the tableau 
$$\overline{T}:=\begin{ytableau}
 j_a & j_{a-1}  & \cdots &j_{k+1} & j_k &  j_{k-1} & \cdots & j_1 \\
i_a & i_{a-1}  & \cdots & i_{k+1} 
\end{ytableau} 
$$
of shape $(a, a-k)$. Clearly, $h_T = (-1)^k f_{\overline{T}}$. Since $T$ is standard,  $\overline{T}$ is standard with 
respect to the inverse order $n \prec n-1 \prec \cdots  \prec k+1 \prec k$ and the letter set $\{k+1, \ldots, n-1, n \}$.  
Hence $\{ h_T \mid T \in W\}$ is linearly independent.

On the other hand, by an argument similar to Lemma~\ref{relation of h_T}, we have 
 $$\sum_{T \in W} c_T' h_T =0.$$ 
It implies  that $c_T'=0$ for all $T \in W$, that is, $\psi'=0$.  
It means that  $\psi \in \frJ$, and we get the expected statement $(\star)$. 
\end{proof}

Combining Theorem~\ref{ISp_{(a,a,1)} is radical} with Theorem~\ref{EGL}, we have the following. 

\begin{cor}\label{ISp_{(a,a,1)} is CM}
If $\chara(K)=0$, the ring $R/\ISp_{(a,a,1)}$ is Cohen-Macaulay. 
\end{cor}

By \cite{WY}, Corollaries~\ref{ISp_{(n-k,k)} is CM} and \ref{ISp_{(a,a,1)} is CM}, we have the following. 

\begin{cor}\label{char(K)=0 main}
Assume that $\chara(K)=0$. 
Then $R/\ISp_\lambda$ is Cohen--Macaulay if and only if $\lambda$ is one of the following form. 
\begin{itemize}
\item[(1)] $\lambda =(n-d, 1, \ldots, 1)$, 
\item[(2)] $\lambda =(n-d,d)$, 
\item[(3)] $\lambda =(a,a,1)$. 
\end{itemize}
\end{cor}

\section{Characteristic free approach to the low dimensional cases}

\begin{prop}\label{I_{n,n-1} is CM} 
For any $K$, $R/\ISp_{(n-2,2)}$ is a Cohen-Macaulay ring with the Hilbert series 
$$\Hilb(R/\ISp_{(n-2,2)},t)=\frac{1+(n-2)t+t^2}{(1-t)^2}.$$
\end{prop}

\begin{proof}
As in  the previous section, set $S:= K[x_1, \ldots, x_{n-1}]$, and $\varphi: R\to S$ the natural surjection. 
Since $x_n$ is a non-zero divisor of $R/\ISp_{(n-2,2)}$,  and $R/(\ISp_{(n-2,2)}+(x_n)) \cong S/\varphi(\ISp_{(n-2,2)})$, it suffices to show that $S/\varphi(\ISp_{(n-2,2)})$ is Cohen--Macaulay. 
In the proof of Theorem~\ref{ISp_{(n-k,k)} is radical}, we have shown that $\varphi(\ISp_{(n-d, d)})$ is radical. 
Since $\varphi(\ISp_{(n-2, 2)})$ and $\ISp_{(n-2,1)}$ are  radical, we have 
$$\varphi(\ISp_{(n-2,2)})= \ISp_{(n-2,1)} \cap I_{\<2\>}$$  
by Lemma~~\ref{radical of frJ}, where $\ISp_{(n-2,1)}$ is a Specht ideal in $S$ and $I_{\<2 \>} \subset S$ is the ideal generated by all degree 2 monomials. Consider the short exact sequence
\begin{equation}\label{MV1}
0 \too S/\varphi(\ISp_{(n-2,2)}) \too S/\ISp_{(n-2,1)}  \oplus S/I_{\<2\>} \too S/(\ISp_{(n-2,1)} + I_{\<2\>})
\too 0.
\end{equation}
It is an elementally fact of the Stanley--Reisner ring theory that $S/I_{\<2\>}$ is a 1-dimensional Cohen--Macaulay ring. Moreover, we have $S/\ISp_{(n-2,1)} \cong K[X]$ and $S/(\ISp_{(n-2,1)} + I_{\<2\>}) \cong K[X]/(X^2)$.  Hence $S/\varphi(\ISp_{(n-2,2)})$ is Cohen--Macaulay by \eqref{MV1}.  

Next we will compute the Hilbert series. By basic techniques of Stanley--Reisner ring theory (see, for example  \cite[pp.212--213]{BH}), 
we have 
$$\Hilb(S/I_{\<2\>},t)= \frac{1+(n-2)t}{(1-t)}.$$
Hence, by \eqref{MV1}, 
\begin{eqnarray*}
\Hilb(S/\varphi(\ISp_{(n-2,2)}),t) &=& \Hilb(S/\ISp_{(n-2,1)},t)+\Hilb(S/I_{\<2\>},t)\\\ 
                              &&  \qquad - \Hilb(S/(\ISp_{(n-2,1)} + I_{\<2\>}),t)\\
 &=& \frac{1}{(1-t)} + \frac{1+(n-2)t}{(1-t)} -(1+t)\\
&=& \frac{1+(n-2)t+t^2}{(1-t)}. 
\end{eqnarray*}
Since $x_n$ is a non-zero divisor of $R/\ISp_{(n-2,2)}$, the Hilbert series of   $R/\ISp_{(n-2,2)}$ has the expected form. 
\end{proof}

\begin{prop}
For all $n \ge 4$, $R/\ISp_{(n-2,2)}$ is Gorenstein. 
\end{prop}

\begin{proof}
By  Proposition~\ref{I_{n,n-1} is CM},  $A:=R/\ISp_{(n-2,2)}$ is a Cohen--Macaulay ring admitting the canonical module 
$\omega_A$, whose  Hilbert series is also  
$$\Hilb(\omega_A,t)=\frac{1+(n-2)t+t^2}{(1-t)^2}.$$
In particular, $\dim_K (\omega_A)_0 =1$. Take $0 \ne a \in   (\omega_A)_0$.  
In the sequel, $\bP_F$ for $F \subset [n]$ denotes the prime ideal 
of $A=R/\ISp_{(n-2,2)}$ given by the quotient of $P_F \subset R$. Since 
$$\Ass_A \omega_A = \Ass_A A =\{ \bP_F \mid F \subset [n], \#F =n-1 \}, $$
there is some $F' \subset [n]$  such that  $\# F'=n-1$ and $\bP_{F'} \supset \ann_A(a)$.  
On the other hand,   the symmetric group $S_n$ also acts on $\omega_A$.  Since $a \in (\omega_A)_0 \cong K$,  $a$ is stable under  the $S_n$-action  up to scalar multiplication.   
Hence $S_n$ also acts on $ \ann_A(a)$,  and we have  $\bP_F \supset \ann_A(a)$  for all $F \subset [n]$ with $\# F=n-1$.  
Since $\bigcap_{\# F=n-1} P_F =I_{n,n-1}=\ISp_{(n-2,2)}$, 
we have $\ann_A(a) \subset  \bigcap_{\# F=n-1} \bP_F=(0)$, and hence $A \cong Aa$. Since $Aa$ and  $\omega_A$ have the same Hilbert functions, $\omega_A=Aa$, and $A$ is Gorenstein. 
\end{proof}

\begin{thm}\label{3-ji shiki}
For $n \ge 5$, $R/I_{n,n-2}$ is Cohen-Macaulay if and only if $\chara(K) \ne 2$.  
Hence, for $n \ge 6$, $R/\ISp_{(n-3,3)}$ is Cohen-Macaulay if and only if $\chara(K) \ne 2$.  
The same is true for $R/\ISp_{(2,2,1)}$. 
\end{thm}

\begin{proof}
By virtue of Theorems~\ref{ISp_{(n-k,k)} is radical} and \ref{ISp_{(a,a,1)} is radical}, the second and third statements follow from the first, so it suffices to show the first. 

As above, let $S:=K[x_1, \ldots, x_{n-1}]$ be the polynomial ring and  $\varphi: R \to S$ the natural surjection. 
Note that  $I_{5, 3} =\ISp_{(2,2,1)}$ and $I_{n,n-2} =\ISp_{(n-3,3)}$ for $n \ge 6$. Set $\mu :=(n-3,2)$. Then $\ISp_\mu \subset S$ is radical by Theorem~\ref{ISp_{(n-k,k)} is radical}, and we have 
$$\varphi(I_{n,n-2}) = \ISp_\mu \cap I_{\<3\>}$$
by Lemma~\ref{radical of frJ}. 
We consider the exact sequence 
\begin{equation}\label{MV3}
0 \too S/\varphi(I_{n,n-2}) \too S/\ISp_\mu  \oplus S/I_{\<3\>} \too S/(\ISp_\mu + I_{\<3\>} )
\too 0.
\end{equation}
Since $\ISp_\mu$ and $I_{\<3\>}$ have no associated prime in common, $A:=S/(\ISp_\mu + I_{\<3\>} )$ has dimension 1. 
Since $S/\ISp_\mu$ and  $S/I_{\<3\>}$  are 2-dimensional Cohen--Macaulay rings, 
$R/I_{n,n-2}$ is Cohen--Macaulay, if and only if so is $S/\varphi(I_{n,n-2})$, if and only if  
so is $A$. So  it suffices to show that $A$ is Cohen--Macaulay if and only if $\chara(K) \ne 2$.  
Now let us analyze the structure of $A$.  

For distinct $i,j,k \in [n-1]$, we have 
$x_ix_jx_k \in I_{\<3\>} \subset  \ISp_\mu + I_{\<3\>}$. 
Since  $(x_i -x_l)(x_j-x_k) \in \ISp_\mu$ for distinct  $i,j,k,l \in [n-1]$, we have  
$$x_i^2x_j -x_i^2x_k = x_i(x_i -x_l)(x_j-x_k) +x_ix_lx_j-x_ix_lx_k \in  \ISp_\mu + I_{\<3\>}.$$
Similarly, considering  $x_ix_j(x_i -x_l)(x_j-x_k)$, we have $x_i^2x_j^2 \in  \ISp_\mu+I_{\<3\>}$. 
So non-zero  monomials of $A_m$ for $m \ge 2$ are of the form $\overline{x_i^m}$ or $\overline{x_i^{m-1}x_j}$.  Moreover,  $\overline{x_i^{m-1}x_j}=\overline{x_i^{m-1}x_k}$, if $i \ne j, k$. 
Hence, for $m \ge 3$, the homogeneous component $A_m$ is spanned by  $2(n-1)$  
elements 
$$\overline{x_1^m}, \ \ldots, \overline{x_{n-1}^m},  \ \overline{x_1^{m-1}x_2}, \   \overline{x_2^{m-1}x_3}, \ \ldots, \ \overline{x_{n-2}^{m-1}x_{n-1}}, \ \overline{x_1x_{n-1}^{m-1}}.$$ 
For simplicity, let 
$$\alpha_1^{(m)}, \ldots, \alpha_{n-1}^{(m)}, \beta_1^{(m)},  \beta_2^{(m)}, \ldots ,  \beta_{n-1}^{(m)}$$
denote these $2(n-1)$ elements.  

To show that $\alpha_1^{(m)}, \ldots, \alpha_{n-1}^{(m)}, \beta_1^{(m)},  \ldots \beta_{n-1}^{(m)}$ are linearly independent, assume that  
$$c_1 \alpha_1^{(m)} + \cdots + c_{n-1}\alpha_{n-1}^{(m)} + d_1 \beta_1^{(m)} +  d_2 \beta_2^{(m)} + \cdots + d_{n-1} \beta_{n-1}^{(m)}=0$$
for $c_1, \ldots, c_{n-1}, d_1, \ldots, d_{n-1} \in K$. Then we have  
$$
c_1 x_1^m +  \cdots + c_{n-1}x_{n-1}^m + d_1 x_1^{m-1}x_2+ \cdots + d_{n-1}x_1x_{n-1}^{m-1} \in \ISp_\mu + I_{\<3\>}.   
$$
Hence there is a degree $m$ element $f \in I_{\<3\>}$ such that 
\begin{equation}\label{linear relation}
c_1 x_1^m +  \cdots + c_{n-1}x_{n-1}^m + d_1 x_1^{m-1}x_2+ \cdots + d_{n-1}x_1x_{n-1}^{m-1} +f \in \ISp_\mu.  
\end{equation}
For any $a \in K$, if we put $x_1 =a$ and $x_i =1$ for all $i \ne 1$ 
then the left side of \eqref{linear relation} becomes 0 (all elements in  $\ISp_\mu$ 
have this property). Hence, for  any $a \in K$, $x_1 =a$ is a root of the equation 
\begin{equation}\label{m-1 ji houteisiki}
c_1 x_1^m + d_1 x_1^{m-1} + (\text{lower degree terms})=0.
\end{equation}
Since $\# K =\infty$, the left side of  \eqref{m-1 ji houteisiki} is the zero polynomial, and $c_1 =d_1 =0$. Similarly, we have $c_i=d_i=0$ for all $i$. It means that $\alpha_1^{(m)}, \ldots, \alpha_{n-1}^{(m)}, \beta_1^{(m)}, \ldots \beta_{n-1}^{(m)}$ are linearly independent. Hence they form a basis of $A_m$, in particular, 
we have $\dim_K A_m=2(n-1)$ for all $m \ge 3$.

\medskip

\noindent \underline{When $\chara(K) \ne 2$:} 
We will show that $e_1 = x_1 + x_2 + \cdots +x_{n-1}$ is $A$-regular. 
Since $e_1$ is clearly  $(S/\ISp_\mu)$-regular, and  $I_{\<3\>}$ is generated by degree 3 elements, we have $e_1 y \ne 0$ for all $0 \ne y \in A_m$ with $m=0,1$. 
Since we have 
$$e_1 \alpha_i^{(m)} = \alpha_i^{(m+1)}+(n-2)\beta_i^{(m+1)}, \qquad e_1 \beta_i^{(m)} =  \beta_i^{(m+1)}$$
for each $1 \le i \le n-1$, it is easy to see that $e_1 y \ne 0$ for all $0 \ne y \in A_m$ with $m \ge 3$. So it remains to show the case $ 0 \ne  y \in A_2$. Consider the $K$-linear map 
$$f: A_2 \ni y \longmapsto e_1 y \in A_3.$$
For  distinct  $i, j \in [n-1]$, we have 
$$f(\overline{x_i^2})= \alpha_i^{(3)}+ (n-2)\beta_i^{(3)}, \qquad f(\overline{x_ix_j})=\beta_i^{(3)} + \beta_j^{(3)}.$$ 
Hence, for distinct $i, j, k \in [n-1]$, we have 
$$f(-\overline{x_jx_k} + \overline{x_kx_i} + \overline{x_i x_j})= -\beta_j^{(3)} -\beta_k^{(3)}+\beta_k^{(3)} + \beta_i^{(3)}
+\beta_i^{(3)} + \beta_j^{(3)}= 2\beta_i^{(3)}.$$
Since  $\chara(K) \ne 2$ now, it follows that  
$\beta_i^{(3)} \in \Im f$,  and hence $\alpha_i^{(3)} \in \Im f$.  So $f$ is surjective.  
On the other hand, we have $A_2=(S/\ISp_\mu)_2$, and the Hilbert series of 
$S/\ISp_\mu$ is given in Proposition~\ref{I_{n,n-1} is CM}  (of course, we should replace $n$ by $n-1$), so we have $$\dim A_2= \dim (S/\ISp_\mu)_2 =2n-2 =\dim A_3.$$
It follows that the surjective map $f$ is also injective, and hence   $e_1 y \ne 0$ for all $0 \ne y \in A_2$. Summing up, $e_1$ is   $A$-regular. 
Since $\dim A=1$, $A$ is Cohen-Macaulay.  

\medskip

\noindent \underline{When $\chara(K) = 2$:} 
Set $y:=\overline{x_1x_2}+\overline{x_2x_3}+\overline{x_3x_1} \in A_2$. Clearly,  $y \ne 0$.  For $i \ge 4$, we have $x_iy=0$. On the other hand, for   $i=1,2,3$, we also have $x_iy= 2\beta_i^{(3)}=0$ now. 
Hence $y$ is a  non-zero socle element, and  $\depth A=0$. It means that $A$ is  {\it not} Cohen--Macaulay. 
\end{proof}

\begin{rem} {\it Macaulay2} shows that the Betti diagram of $R/\ISp_{(3,3)} \, (=R/I_{6,4})$ is  
\begin{verbatim}
total: 1 5 9 5
          0: 1 . . .
          1: . . . .
          2: . 5 . .
          3: . . 9 5
\end{verbatim}
if $\chara(K) =0$ (actually, if $\chara(K) \ne 2$), and  
\begin{verbatim}
 total: 1 5 9 6 1
          0: 1 . . . .
          1: . . . . .
          2: . 5 . . .
          3: . . 9 5 1
          4: . . . 1 .
\end{verbatim}
if $\chara(K) = 2$. 
\end{rem}

The computer experiments suggest the following conjectures. We have to say that the computation of Specht ideals is very heavy, so we do not have so much experience.  

\begin{conj}
Let $\lambda=(\lambda_1, \ldots, \lambda_l)$ be a partition of $n$ satisfying the condition (2) or (3) of Proposition~\ref{CM ISP}. Then $R/\ISp_\lambda$ is Coehn--Macaulay if and only if $\chara(K)=0$ or $\chara(K) \ge n-\lambda_1$. 
\end{conj}

\begin{conj}
If $\chara(K)=0$, $R/\ISp_{(a,a,1)}$ has an $(a+2)$-linear resolution. 
\end{conj}

\section*{Acknowledgements} 
This project first started as a collaboration with Professor Junzo Watanabe. However, we peacefully dissolved the collaboration, since we found that each of us were interested in  different aspects of this topic. 
I would like to thank him for valuable comments and discussion （for example, he told me Lemma~\ref{radical reduction}). 

I am also grateful to Kenshi Okiyama, who was my undergraduate student, for indispensable help to  computer experiment in the beginning of this study.   I also thank Professor Satoshi Matsumoto of Tokai university for    computer experiment in the later stage. 
I am also grateful to Professor Mitsuyasu Hashimoto for valuable  comments.

\end{document}